\documentclass[final]{siamltex}
\usepackage{epsfig,psfrag,epsf}
\usepackage{amsmath,amsfonts}

\title{A weighted minimum gradient problem with complete electrode model boundary conditions
 for conductivity imaging\thanks{Received by the editors 2015.}}

\author{Adrian Nachman\thanks{Department of Mathematics,
        Edward S Roger Sr. Department of Electrical and Computer Engineering, and Institute of Biomaterials and Biomedical Engineering, University of Toronto,
        Toronto, Ontario, Canada ({\tt nachman@math.toronto.edu}).}\and
        Alexandru Tamasan\thanks{Department of Mathematics, University
        of Central Florida, Orlando, FL, USA ({\tt tamasan@math.ucf.edu}).}
        \and Johann Veras\thanks{ Lockheed Martin Co., Sand Lake Rd, MP 450, Orlando, FL, 32819, USA ({\tt Jveras@knights.ucf.edu}).}}

\begin{document}

\maketitle

\begin{abstract}
We consider the inverse  problem of recovering an
isotropic electrical conductivity from interior knowledge of the magnitude of
one current density field generated by applying current on a set of electrodes. The required interior data can be obtained by means of MRI measurements. On the boundary we only require knowledge of the electrodes, their impedances, and the corresponding average input currents. From the mathematical point of view, this practical question leads us to consider a new weighted minimum gradient problem for functions satisfying the boundary conditions coming from the Complete Electrode Model (CEM) of Somersalo, Cheney and Isaacson. We show that this variational problem has non-unique solutions. The surprising discovery is that the physical data is still sufficient to determine the geometry of (the connected components of) the level sets of the minimizers. We thus obtain an interesting phase retrieval result: knowledge of the input current at the boundary allows determination of the full current vector field from its magnitude.
We characterize locally the non-uniqueness in the variational problem.
In two and three dimensions we also show that additional measurements of the voltage potential along a curve joining the electrodes yield unique determination of the conductivity. The proofs involve a maximum principle and a new regularity up to
the boundary result for the CEM boundary conditions. A nonlinear algorithm is proposed and implemented to illustrate the theoretical results.
\end{abstract}



\pagestyle{myheadings} \thispagestyle{plain} \markboth{A. Nachman, A. Tamasan and J. Veras}{
A weighted minimum gradient problem in conductivity imaging}
\section{Introduction}
\label{intro}
We consider the inverse problem of reconstructing an inhomogenuous isotropic
electrical conductivity $\sigma$ in a domain $\Omega\subset\mathbb{R}^n$, $n\geq2$, from interior knowledge of the magnitude $a$ of one
current density field and of corresponding boundary data.

Most of the existing results on this problem (see a brief survey of previous work at the end of this introduction) consider Dirichlet boundary conditions. In this paper we study boundary conditions which model what can actually be measured in practical experiments.  We work with the beautiful Complete Electrode Model (CEM) originally introduced in \cite{somersaloCheneyIsaacson} and shown to best describe the physical data: For $k=0,...,N$, let $e_k\subset \partial\Omega$ denote the surface electrode of constant impedance $z_k$ through which one injects a net current $I_k$.

The CEM assumes the voltage potential $u$ inside and the constant voltages $U_k$'s on the surface of the electrodes distribute according to the boundary value problem
\begin{align}
&\nabla\cdot\sigma\nabla u=0, \quad\text{in}\;\Omega,\label{conductivity_eq}\\
&u+z_k\sigma\frac{\partial u}{\partial\nu}= U_k\quad \text{on}\; e_k,\; \;\text{for}\; k=0,...,N,\label{robin_kForward} \\
&\int_{e_k} \sigma\frac{\partial u}{\partial\nu}ds = I_k,\quad \text{for}\; k=0,...,N,\label{inject_kForward}\\
&\frac{\partial u}{\partial\nu}=0, \quad\text{on}\; \partial\Omega\setminus\bigcup_{k=0}^{N}e_k, \label{no_flux_off_electrodesForward}
\end{align}where $\nu$ is the outer unit normal.
If a solution  exists, an integration of \eqref{conductivity_eq} over $\Omega$ together with \eqref{inject_kForward} and \eqref{no_flux_off_electrodesForward} show that
\begin{align}\label{sum0inject}
\sum_{k=0}^N I_k=0
\end{align}
is necessary. Physically, the zero sum of the boundary currents account for the absence of sources/sinks of charges.
The constants $U_k$ appearing in \eqref{robin_kForward} represent \emph{unknown} voltages on the surface of
the electrodes, and the difference from  the traces $u|_{e_k}$
of the interior voltage potential governs the flux of the current through the skin to the electrode.
We refer to the problem \eqref{conductivity_eq}, \eqref{robin_kForward},
\eqref{inject_kForward}, and \eqref{no_flux_off_electrodesForward} as the \emph{forward problem}.

Under the assumptions that  $\Omega$ is a Lipschitz domain, the conductivity is essentially bounded
with real part bounded away from zero, the electrodes $e_k$ are (relatively) open connected subsets
of $\partial\Omega$ whose closure are disjoint,
the impedances $z_k$ have positive real part, and the injected currents $I_k$ satisfy \eqref{sum0inject}, the forward problem has a unique solution
$(u; \langle U_0,...,U_N\rangle)\in H^1(\Omega)\times \mathbb{C}^{N+1}$ up to an additive constant,
as shown in \cite{somersaloCheneyIsaacson}. We normalize this constant by imposing the electrode voltages ${U}=\langle U_0,...,U_N\rangle$ to lie in the hyperplane
\begin{align}\label{normalization}
\Pi:=\{U\in\mathbb{R}^{N+1}:\;U_0+...+U_N=0\}.
\end{align}
The net input currents $I_k$, $k=0,...,N$ as in \eqref{sum0inject}, generate a  current density field $J=-\sigma\nabla u$, where $(u,U)\in H^1(\Omega)\times \Pi$ is the solution of the forward problem.

In this paper we consider the inverse problem of determining $\sigma$, given the magnitude
\begin{align}\label{a}
a=|\sigma \nabla u|
\end{align}
of the current density field inside $\Omega$.

The conductivity $\sigma$ is \emph{unknown} but assumed real valued
 and satisfying
\begin{align}\label{lowerbound_sigma}
\text{essinf}_{\Omega} \sigma(x)>0.
\end{align}

Each electrode $e_k\subset\partial\Omega$, $k=0,...,N,$ is a known Lipschitz domain subset of the boundary. The surface impedances $z_k$ are assumed real valued. In general we allow them to be inhomogenous functions on the electrodes satisfying
\begin{align}\label{common_lower_impedance}
\text{essinf}_{e_k}{z_k} >0,\quad\text{for }k=0,...,N.
\end{align}
Further smoothness conditions will be assumed for some of the  results in this paper.

We note that, in practice, interior measurements  of all three components of the current density $J$ can be obtained from three  magnetic resonance scans involving two rotations of the object \cite{joy}. However recent engineering advances in ultra-low field magnetic resonance may be used to recover $J$ without rotating the object \cite{nieminenETal}. We hope that the results presented here may lead to further experimental progress on easier ways to measure directly just the magnitude of the current.

We start by remarking that there is non-uniqueness in the inverse problem stated above, as can be seen
in the following example: Let $\Omega=(0,1)\times(0,1)$  be the unit square. We inject
the current $I_1=1$ through the top electrode $e_1=\{(x,1): \;0\leq x\leq 1\}$ of impedance $z_1>0$, ``extract''
the current $I_0=-1$ through the bottom electrode $e_0=\{(0,x):\;0\leq x\leq 1\}$ of impedance $z_0=z_1+1$,
and measure the magnitude $a\equiv 1$ of the current density field in $\Omega$.
Then, for every $\varphi:[0,1]\to [\varphi(0),\varphi(1)]$ an increasing Lipschitz continuous function,
satisfying $\varphi(0)+\varphi(1)=1$, the function $u_\varphi(x,y):=\varphi(y)$ solves the forward problem
\eqref{conductivity_eq}, \eqref{robin_kForward}, \eqref{inject_kForward}, and \eqref{no_flux_off_electrodesForward}
corresponding to a conductivity $\sigma_\varphi(x,y)=1/\varphi'(y)$, yet the magnitudes of the corresponding current densities yield the same interior measurements $\sigma|\nabla u|=\sigma_\varphi|\nabla u_\varphi|\equiv 1$.

More generally, if $(u,U)\in H^1(\Omega)\times\Pi$ is the solution of the forward problem for some $\sigma$, let $\varphi\in Lip(u(\overline\Omega))$ be any Lipschitz-continuous increasing function of one variable, such that $\varphi(t)=t+ c_k$ whenever $t\in u(e_k)$, for each $k=0,...,N $, and constants $c_k$ satisfying $\sum_{k=0}^N c_k=0$. One can easily verify that the function
\begin{align}\label{Lipschitz}
u_\varphi:=\varphi\circ u
\end{align}solves the forward problem with the conductivity
\begin{align}\label{111}
\sigma_\varphi:=\frac{\sigma}{\varphi'\circ u},
\end{align}while $\sigma|\nabla u|=\sigma_\varphi|\nabla u_\varphi|$.

For H\"{o}lder-continuous conductivities, in Theorem \ref{main} we prove
that \eqref{Lipschitz}, \eqref{111} must hold in a neighborhood of any non-critical point. The following example shows that \eqref{Lipschitz}, \eqref{111} need not hold in the whole domain $\Omega$ for a single function $\varphi$.

Let $\Omega\subset R^2 $ be the curvilinear octagon with hyperbolic sides obtained  from the unit disc,
by carving out the peripheral regions along the branches of the hyperbolas $x^2-y^2=\pm\frac{1}{2}$,
and $xy=\pm\frac{\sqrt {3}}{4}$.  The electrodes $e_0, e_2$ are defined respectively by each connected
component of $x^2-y^2=\frac{1}{2}$ inside the disc, whereas $e_1,e_3$ are defined respectively by the connected components of $x^2-y^2=-\frac{1}{2}$. Thus defined, all the electrodes have equal length, denoted by $l$. We assume constant impedances $z_0=z_1=z_2=z_3=1$.  Through $e_0$ and $e_2$ one inputs the net currents $I_0=I_2=\frac{l}{2}$, and extracts $I_1=I_3= -\frac{l}{2}$ through $e_1$ and $e_3$. In $\Omega$ one measures $|J(x,y)|=2\sqrt{x^2+y^2}$. It is easy to check that the constant conductivity $\sigma=1$ is a possible solution to the inverse problem, with the corresponding voltage $u(x,y)=x^2-y^2$. However, this is not the only possibility: Let
$$\omega_\pm:=\{(x,y)\in\Omega:~ \epsilon<x^2-y^2<\frac{1}{2}-\epsilon,~\pm x>0\},$$
and  $\varphi_\pm:(0,\frac{1}{2})\to(0,\frac{1}{2})$ be any two increasing Lipschitz continuous functions with $\varphi_\pm(t)=t$ in $(0,\epsilon)\cup(\frac{1}{2}-\epsilon, \frac{1}{2}) $, for some $\epsilon>0$ sufficiently small.
Define a new conductivity by
\begin{equation}
\tilde\sigma=
\left\{\begin{array}{ll}
\frac{\sigma}{\varphi_\pm'\circ u},& \text{in } \omega_\pm,\\
1,& \text{in } \Omega\setminus\omega_\pm.
\end{array}\right.
\end{equation}
As in the previous examples, one can check that $\tilde\sigma$ is also a solution of the inverse problem for the same
$|J(x,y)|=2\sqrt{x^2+y^2}$, with corresponding potential $v$ equal to $u$ on $\Omega\setminus\omega_\pm$, and equal to $\varphi_\pm\circ u$  on $\omega_\pm$, respectively.

Similar to the approach in \cite{NTT09}, we formulate the inverse problem in terms of a weighted minimum gradient problem. Here we need a functional which is appropriate for the boundary conditions coming from the Complete Electrode Model and is entirely defined in terms of the data in the inverse problem. To obtain such a functional, we found it necessary to revisit the forward problem and recast
it as a minimization problem, see Appendix A. We are then able to show that the solution $(u,U)\in H^1(\Omega)\times\Pi$ of the forward problem is a global minimizer of the functional
\begin{align}\label{1LaplFunctional}
G_a(v,V)=\int_\Omega a|\nabla v|dx+\sum_{k=0}^N\int_{e_k}\frac{1}{2z_k}(v-V_k)^2ds-\sum_{k=0}^N I_k V_k,
\end{align}over $H^1(\Omega)\times \Pi$, with $a=\sigma|\nabla u|$ as in \eqref{a}.

Using $G_a$, we found the surprising fact that, given the positions and impedances of the electrodes,
knowledge of the magnitude of one current density field and of the corresponding average applied currents
(just one number in the case of two electrodes!) is still sufficient to determine the geometry of the connected
components of the equipotential sets. Furthermore, we remark that, since we recover the direction $\overrightarrow{N}$
(including orientation) of the electric field $\nabla u$, we also obtain an interesting phase retrieval result: the full current density vector field $J=a \overrightarrow{N}$ is recovered from its magnitude $a$, and knowledge of the input currents $I_k$ on the surface electrodes, even though the conductivity is not uniquely determined.

Uniqueness of the conductivity can be restored by additional measurement of the voltage potential on a curve $\Gamma$ connecting the electrodes; see Theorem \ref{unique_determination}. The additional measurement involves only a one dimensional subset of boundary measurements,
much less than required in the existing results for the Dirichlet problem.

For the unique determination result we assume that the curve $\Gamma$ and the set of electrodes satisfy a topological assumption, see \eqref{topASS}.

To illustrate the theoretical results, we propose an iterative algorithm and perform a
numerical experiment, see Section \ref{algorithm}. Similar to the algorithm in \cite{NTT09}
we decrease the functional $G_a$ on a sequence of solutions of forward problems for updated conductivities.

Conductivity imaging using the interior knowledge of the magnitude of current densities was first introduced in \cite{seo}. The examples of non-existence and non-uniqueness for the ensuing Neumann problem lead the authors of \cite{seo} to consider the magnitudes of two currents. The possibility of conductivity imaging via the magnitude of just one current density field was shown in \cite{NTT07} via the Cauchy problem, and in \cite{NTT09,NTT11} via  a minimum gradient problem with Dirichlet boundary conditions. Existence and uniqueness of such weighted gradient problems was studied in \cite{jerrard99}. Extensions to the case of inclusions with zero or infinite conductivity were obtained in \cite{MNTa_SIAM,MNTa_BV}. A structural stability result for the minimization problem can be found in \cite{nashedTa10}. Reconstruction algorithms based on the minimization problem were proposed in \cite{NTT09} and \cite{MNT}, and based on level set methods in \cite{NTT07,NTT09,verasTamasanTimonov14}. A local H\"{o}lder- continuous dependence of $\sigma$ on $|J|$ (for unperturbed Dirichlet data) has been recently established in \cite{montaltoStefanov}.  For further references on determining the isotropic conductivity based on measurements of current densities see \cite{zhang,seo,seo_ieee, jeun-rock, nachman,lee,  kimLee}, and
for some results on anisotropic  conductivities see \cite{MaDeMonteNachmanElsaidJoy,MNNick2014, balGuoMonard2d, balGuoMonard3d}.

\section{A weighted minimum gradient problem for the CEM boundary conditions}\label{section:minproblem}

In this section we show that the solution of the forward problem is a global minimizer of the functional $G_a$ in \eqref{1LaplFunctional} over $H^1(\Omega)\times\Pi$. The regularity assumptions are the ones from the forward problem.

\begin{proposition}\label{G_aEquivalence}Let $\Omega\subset \mathbb{R}^n$, $n\geq 2$, be a bounded domain with Lipschitz boundary. Let
 $e_k$,  for $k=0,...,N,$  be disjoint  subsets of the boundary of positive  surface measure, $\sigma\in L^\infty(\Omega)$ and  $z_k\in L^\infty(e_k)$ be bounded away from zero, and $I_k$ satisfy \eqref{sum0inject}. Let $(u,U)\in H^1(\Omega)\times \Pi$ be the unique solution of the forward problem \eqref{conductivity_eq}, \eqref{robin_kForward}, \eqref{inject_kForward}, \eqref{no_flux_off_electrodesForward}. If $a:=\sigma|\nabla u|$, then
\begin{align}\label{globalMIN}
G_a(v,V)\geq G_a(u,U), \quad \forall (v,V)\in H^1(\Omega)\times\Pi.
\end{align}
\end{proposition}

\begin{proof}For any $(v,V)\in H^1(\Omega)\times\Pi$, we have the inequality

\begin{align}
G_a(v,V)&=\int_\Omega\sigma|\nabla u||\nabla v|dx +\frac{1}{2}\sum_{k=0}^N
\int_{e_k}\frac{1}{z_k}(v-V_k)^2ds-\sum_{k-0}^N I_kV_k\nonumber \\
&\geq \int_\Omega \sigma\nabla u \cdot\nabla v dx +\sum_{k=0}^N\int_{e_k}\left[\frac{1}{2z_k}(v-V_k)^2-V_k\sigma\frac{\partial u}{\partial \nu}\right]ds\nonumber\\
&=\int_{\partial\Omega} v\sigma\frac{\partial u}{\partial\nu}ds+\sum_{k=0}^N\int_{e_k}\left[\frac{1}{2z_k}(v-V_k)^2-V_k\sigma\frac{\partial u}{\partial \nu}\right]ds\nonumber\\
&=\sum_{k=0}^N\int_{e_k}\left[v\sigma\frac{\partial u}{\partial\nu}+\frac{1}{2z_k}(v-V_k)^2 -V_k\sigma\frac{\partial u}{\partial \nu}\right]ds\nonumber\\
&=\sum_{k=0}^N\int_{e_k}\frac{1}{z_k}\left[(v-V_k)(U_k-u)+\frac{1}{2}(v-V_k)^2\right]ds,\label{main_estimate}
\end{align}where the first equality uses \eqref{a}, the next line uses \eqref{inject_kForward} and the Cauchy-Schwarz inequality, the next equality uses \eqref{conductivity_eq} and the divergence theorem, the third equality uses \eqref{no_flux_off_electrodesForward}, and the last equality uses \eqref{robin_kForward}.

In particular, when $v=u$, the inequality in the estimate \eqref{main_estimate} holds with equality yielding
\begin{align}
G_a(u,V)=\sum_{k=0}^N\int_{e_k}\frac{1}{z_k}\left[\frac{1}{2}(u-V_k)^2-(u-V_k)(u-U_k)\right]ds,
\end{align}and
\begin{align}\label{minimumValue}
G_a(u,U)=-\frac{1}{2}\sum_{k=0}^N\int_{e_k}\frac{1}{z_k}(u-U_k)^2ds.
\end{align}

The global minimizing property \eqref{globalMIN} then follows from \eqref{main_estimate} and \eqref{minimumValue} using the pointwise inequality
\begin{align}\label{quadratic}
\frac{1}{2}(v-V_k)^2-(v-V_k)(u-U_k)
\geq -\frac{1}{2}(u-U_k)^2.
\end{align}

\end{proof}

Note how each of the CEM equations \eqref{conductivity_eq}, \eqref{robin_kForward}, \eqref{inject_kForward}, \eqref{no_flux_off_electrodesForward} was used in the above proof.


\section{Local characterization of non-uniqueness and applications}

In this section we state and prove our main result and its consequences to the conductivity imaging problem.

\begin{theorem}\label{main}
Let $\Omega\subset\mathbb{R}^d$, $d\geq 2$ be a bounded, Lipschitz domain,  and let $e_k$, $k=0,...,N$,  be disjoint  subsets of the boundary of positive surface measure. Assume that the corresponding  impedances $z_k$ satisfy \eqref{common_lower_impedance}, and that the given currents  $I_k$ are such that \eqref{sum0inject} holds. Let $(u, U),(v,V)\in H^1(\Omega)\times\Pi$, be the solutions of the forward problem \eqref{conductivity_eq}, \eqref{robin_kForward}, \eqref{inject_kForward}, and \eqref{no_flux_off_electrodesForward} corresponding to unknown conductivities $\sigma,\tilde\sigma\in C^\alpha(\Omega)$ satisfying \eqref{lowerbound_sigma}. Assume that
\begin{align}\label{same_|J|}
\sigma|\nabla u|=\tilde\sigma|\nabla v |>0\;a.e. \mbox{in}\;\Omega.
\end{align}

Then:

(i) for each $k=0,...,N$,
\begin{align}\label{on_the_electrodes}
u|_{e_k}-U_k = v|_{e_k}-V_k, \;\text{a.e. on }e_k.
\end{align}

(ii) for a.e. $x^0\in\Omega$ with $|\nabla u(x_0)|\neq 0$, there exists a neighborhood $O_0$ of $x^0$ and a
function $\varphi\in C^1(u(O_0))$, such that
\begin{align}\label{charact}
v=\varphi\circ u,\quad \text{in }O_0,
\end{align}and
\begin{align}\label{tilde_sigma}
\tilde\sigma=\frac{\sigma}{\varphi'\circ u},\quad \text{in } O_0.
\end{align}Moreover, the equalities \eqref{charact} and \eqref{tilde_sigma} hold
on the union of the connected components of the level sets of $u$ passing through $O_0$.

\end{theorem}

\begin{proof}
From the interior elliptic regularity we have $u, v\in C^{1,\alpha}(\Omega)$;
see, e.g., \cite[Theorem 8.34]{gilbargTrudinger}. Thus, the sets of critical points of $u$ and $v$ are closed,
and, by hypothesis \eqref{same_|J|} they are negligible. Let $S$ denote their union. It follows that $\Omega\setminus S$ is open and dense in $\Omega$.

According to Proposition \ref{G_aEquivalence} $(u,U)$ and $(v,V)$ are both minimizers of $G_a$, and thus
\begin{align}\label{sameValue}
G_a(u,U)=G_a(v,V).
\end{align}
In particular, using \eqref{minimumValue}, \eqref{sameValue}, \eqref{main_estimate}, and \eqref{quadratic}, the inequalities
\begin{align*}
-\frac{1}{2}\sum_{k=0}^N\int_{e_k}\frac{1}{z_k}(u-U_k)^2ds&=G_a(v,V)\\
&\geq\sum_{k=0}^N\int_{e_k}\frac{1}{z_k}\left[(v-V_k)(U_k-u)+\frac{1}{2}(v-V_k)^2\right]ds\\
&\geq -\frac{1}{2}\sum_{k=0}^N\int_{e_k}\frac{1}{z_k}(u-U_k)^2ds
\end{align*}must be equalities. Thus,
\begin{align}\label{sumOfZeros}
\sum_{k=0}^N\int_{e_k}\frac{1}{2z_k}\left\{(v-V_k) - (u-U_k)\right \}^2ds=0.
\end{align}
Now \eqref{sumOfZeros} together with $z_k>0$  yield \eqref{on_the_electrodes} for each  $k=0,1,...,N.$

We will prove part (ii) for any point in $\Omega\setminus S$.

Since \eqref{main_estimate} holds with equality, we must also have
\begin{align}\label{parallel0}
\nabla u\cdot\nabla v=|\nabla u|\cdot|\nabla v|, \quad\text{a.e. on }\Omega.
\end{align}
Since both gradients are continuous, in view of \eqref{parallel0} they must be parallel whenever one of them is nonzero, in particular
\begin{align}\label{parallel}
\nabla v=\mu\nabla u\text{ in }\Omega\setminus S,
\end{align}for some $\mu\in C(\Omega\setminus S)$ with $\mu > 0$ in $\Omega\setminus S$. Moreover, since $\Omega\setminus S$ is dense, $\mu$ extends by continuity to the whole domain $\Omega$.

Let $x^0\in\Omega\setminus S$ be arbitrarily fixed. For some component $x_j$,  $\left|\frac{\partial u}{\partial x_j}(x^0)\right|>0$.
Consider $G=(G^1,...,G^n): \Omega\to \mathbb{R}^n$ defined by $G^k(x_1,...,x_n)=x_k$ if $k\neq j$,
and $G^j(x_1,...,x_n) =u(x_1,...,x_n)$. Then the Jacobian determinant
$$|DG(x^0)|=\left|\frac{\partial u}{\partial x_j}(x^0)\right|>0,$$and,
by the inverse function theorem, there is a neighborhood $O_0$ of $x_0$ in $\Omega\setminus S$, such that $G:O_0\to G(O_0)$ is a diffeomorphism.

Consider  $v\circ G^{-1}: G(O_0)\to\mathbb{R}$. We use \eqref{parallel} to show that $$G(O_0)\ni y= (y_1,....,y_n)\mapsto v\circ G^{-1}(y)$$ is independent of $y_k$, for all $k\neq j$. Indeed, for $y\in G(O_0)$,
\begin{align*}
\frac{\partial}{\partial y_k}(v\circ G^{-1})(y)&=\sum_{i=1}^n\frac{\partial v}{\partial x_i}(G^{-1}(y))\frac{\partial x_i}{\partial y_k}(y)\\
&=\mu(G^{-1}(y))\sum_{i=1}^n\frac{\partial u}{\partial x_i}( G^{-1}(y))\frac{\partial x_i}{\partial y_k}(y)\\
&=\mu(G^{-1}(y))\frac{\partial }{\partial y_k}(u\circ G^{-1})(y)\\
&=\mu(G^{-1}(y))\frac{\partial y_j }{\partial y_k}=0,\quad\text{for }k\neq j,
\end{align*}where the second equality uses \eqref{parallel} and the next to the last equality uses the definition of $G$.

For each $y_j\in v(O_0)$ we can now well define
\begin{align}\label{def_varphi}
\varphi(y_j):=v\circ G^{-1}(y_1,...,y_n).
\end{align}If $x=G^{-1}(y)\in O_0$, then $y_j= G^j(x) =u(x)$, and the equation above shows
$$\varphi(u(x))= v(x),\quad\text{for all }x\in O_0.$$



By \eqref{same_|J|} we also obtain
\begin{align*}
\sigma(x)|\nabla u(x)|=\tilde\sigma(x)|\nabla v(x)|=\tilde\sigma(x)\varphi'(u(x))|\nabla u(x)|,\quad \forall x\in O_0.
\end{align*}Since $|\nabla u|>0$ in $ O_0$, the relation \eqref{tilde_sigma} follows.

Finally, since $u$ and $v$ are constant on each other's connected components of level sets within $\Omega\setminus S$ (as can be seen from a differentiation in the direction tangential to the level set and \eqref {parallel}), the identity \eqref{charact} extends to points on any connected component of a level set of $u$ passing through $O_0$.


\end{proof}

The result above  implies that knowledge of the input currents at the boundary is sufficient to determine the full current density from measurements of its magnitude in the interior, even when the conductivity is not determined uniquely.  Thus we have obtained the following ``phase retrieval" result, which may be of independent interest.

\begin{corollary}[Phase retrieval]\label{phase} Let $\Omega\subset\mathbb{R}^d$, $d\geq 2$ be a bounded, Lipschitz domain,  and let $e_k$, $k=0,...,N$,  be disjoint  subsets of the boundary of positive surface measure. Assume that the corresponding  impedances $z_k$ satisfy \eqref{common_lower_impedance}, and that the given currents  $I_k$ are such that \eqref{sum0inject} holds.
Let $(u, U),(v,V)\in H^1(\Omega)\times\Pi$ be the solutions of the forward problem \eqref{conductivity_eq}, \eqref{robin_kForward}, \eqref{inject_kForward}, and \eqref{no_flux_off_electrodesForward} corresponding to unknown conductivities $\sigma,\tilde\sigma\in C^\alpha(\Omega)$ satisfying \eqref{lowerbound_sigma}.
 Let $J:=\sigma\nabla u$ and $\tilde{J}:=\tilde\sigma \nabla v$ be the corresponding current densities. If
\begin{align}
|J|=|\tilde{J}| >0\;a.e. \mbox{in}\;\Omega,
\end{align}then
\begin{align}
J=\tilde{J}\;  \mbox{in}\;\Omega.
\end{align}
\end{corollary}
\begin{proof} For $x^0\in\Omega\setminus S$ arbitrarily fixed, let $O_0$ and
$\varphi:~v(O_0) \to \mathbb{R}$ be the function provided by Theorem \ref{main}. From \eqref{tilde_sigma} and \eqref{charact} we have for all $x\in O_0$,
$$\tilde{J}(x)=\tilde{\sigma}(x)\nabla v(x)=\frac{\sigma(x)}{\varphi'\circ u(x)}\nabla v(x)=\sigma(x)\nabla u(x)=J(x).$$ Since $x^0$ is arbitrary in $\Omega\setminus S$,  $J$ and $\tilde{J}$ are continuous in $\Omega$, and $\Omega\setminus S$ is dense in $\Omega$, the result follows.

\end{proof}

{\bf Remark:} Theorem \ref{main} part (ii), and Corollary \ref{phase} are also valid for the Neumann problem.
The proofs are essentially the same as above, using instead the functional
$G^N_a(v)=\int_\Omega a|\nabla v|dx-\int_{\partial\Omega}fvds \displaystyle$,
where $f=\left.\sigma\frac{\partial u}{\partial\nu}\right|_{\partial\Omega}$ is the imposed current on the boundary.

\section{Unique determination for two and three dimensional models}
To determine the conductivity uniquely, we will assume additional knowledge of the voltage potential on some part of the boundary. For brevity, let
$E:=\bigcup_{k=0}^N e_k\displaystyle$ denote the set of all electrodes.

We show below that knowledge of the voltage potential $u$
along a boundary curve $\Gamma\subset\partial\Omega\setminus E$ which joins the electrodes is
sufficient to yield uniqueness in two and three dimensional domains. We assume that $\Gamma$ satisfies the
topological assumption:
\begin{equation}\label{topASS}
\Gamma\cup E \mbox{ is a connected set and each connected component of } \partial\Omega \setminus (\Gamma\cup {E}) \mbox{ is simply connected}.
\end{equation}
In two dimensions the second assumption is trivially satisfied.

For the proof of the uniqueness result we need continuity of solutions up to the
boundary for $\sigma$-harmonic functions satisfying CEM boundary conditions.

As a direct consequence of the Proposition \ref{regularity} and the Sobolev embedding theorem we have the following.
\begin{corollary}\label{holderRegularity}Let $\Omega\subset \mathbb{R}^n$ be a $C^2$-domain, $n=2,3$, and
let $\sigma\in C^\alpha(\overline\Omega)$ be $C^2$-smooth near the boundary.
Assume that the electrodes $e_k$ has Lipschitz boundaries, and $z_k\in C^2(e_k)$, $k=0,...,N$. Let $(u,U)\in H^1(\Omega)\times\Pi$  be the solution to the forward problem. Then

a)  $u\in C^{\alpha}(\overline\Omega)$, for $0<\alpha<1/2$.

b) For any $x_0\in\partial\Omega\setminus\partial E$, there is a neighborhood $V_0\subset \overline\Omega$ of $x_0$,  such that $u\in C^{1,\alpha}(V_0)$,  for $0<\alpha<1/2$.
\end{corollary}

Another idea in the uniqueness result below is that the range of $u$ on
the union of the curve $\Gamma$ and the electrodes is the same as the range of
$u$ in $\overline\Omega$. This follows from the following maximum principle for the CEM,
which may be of independent interest.

\begin{proposition}(Maximum principle for CEM)\label{MAX}
Under the smoothness assumptions in the Corollary \ref{holderRegularity},
the solution $u$ achieves its minimum and maximum on $\overline{E}$.
If $\Gamma\subset\partial\Omega\setminus E$ is a curve such that $\Gamma\cup E$ is connected, then the range of $u$ over $\Gamma\cup \overline{E}$ coincides with the range of $u$ over $\overline\Omega$.
\end{proposition}

\begin{proof}
By the weak maximum principle, the maximum $M$  and minimum $m$ of $u$ over $\overline\Omega$ occur on the boundary. By Hopf's strong maximum principle, at  a point of maximum, say $x_0\in\partial\Omega$, the normal derivative $\frac{\partial u}{\partial\nu}(x_0)$ must be strictly positive.  From the boundary condition \eqref{no_flux_off_electrodesForward} we then deduce $x_0\in \overline{E}$. The same argument applies to a point of minimum, where the normal derivative is strictly negative. Since $\Gamma\cup E$ is connected, and $u$ is continuous on $\overline\Omega$, $u(\Gamma\cup \overline{E})$ is a closed interval. Since the maximum $M$ and minimum $m$ occur on $\overline{E}$, then the range $u(\Gamma\cup \overline{E})=[m,M]=u(\overline\Omega)$.

\end{proof}

We are now ready to  prove our main uniqueness result.

\begin{theorem} [Unique determination]\label{unique_determination}
Let $\Omega\subset \mathbb{R}^n$ be a $C^2$-domain, $n=2,3$. Assume that
the electrodes have Lipschitz boundaries,
and $z_k\in C^2(e_k)$, $k=0,...,N$. For the currents $I_0,...,I_N$
satisfying \eqref{sum0inject}, let  $(u, U),(v,V)\in H^1(\Omega)\times\Pi$ be the solutions
of the forward problem \eqref{conductivity_eq}, \eqref{robin_kForward}, \eqref{inject_kForward} and \eqref{no_flux_off_electrodesForward} corresponding to unknown conductivities $\sigma,\tilde\sigma\in C^\alpha(\overline\Omega)$, which are assumed $C^2$-smooth near the boundary
and satisfying \eqref{lowerbound_sigma}.

If
\begin{align}
&\sigma |\nabla u|=\tilde\sigma |\nabla \tilde{u}|>0,\;a.e.\;\text{in}\;\Omega,\label{common|J|}\quad\text{and}\\
&u|_\Gamma=\tilde{u}|_\Gamma,\label{voltage_measurement}
\end{align}then
\begin{align}
&u=\tilde{u}\;\text{in }\overline\Omega,\label{u=tildeu+C}\\
&\sigma=\tilde{\sigma}\;\text{in }\Omega.
\end{align}
\end{theorem}

\begin{proof} We will give the proof for the three dimensional case and indicate where arguments simplify in the two dimensional case.

From Corollary \ref{holderRegularity} we know that $u$ and $\tilde{u}$ are continuous up to the boundary. In particular, the identity \eqref{on_the_electrodes} in Theorem \ref{main} shows that $u=\tilde{u}+c_k$ on each electrode $e_k$, $k=0,...,N$. Since $u=\tilde{u}$ on $\Gamma\cap \partial E$ by hypothesis \eqref{voltage_measurement}, we conclude that $c_k=0$ for each $k=0,...,N$.
So far we showed that $u$ and $\tilde{u}$ coincide on $\Gamma\cup {E}$, and following Proposition \ref{MAX}, $u(\overline\Omega)=\tilde{u}(\overline\Omega)=:[m,M]$.

We refer to a value $t\in [m,M]$ as being regular if the corresponding $t$-level  set is free of singular points.
For $t$-regular, let $\Sigma_t$ be a connected component of the $t$- level
set. The arguments in \cite[Theorem 1.3]{NTT09} showing that $\Sigma_t$ reaches $\partial\Omega$ do not use any boundary information, and thus they remain valid for the CEM boundary conditions; we recall them in Appendix B for completeness.

To prove unique determination, it now suffices to show that  $\Sigma_t$
intersects $\Gamma\cup {E}$. We reason by contradiction: Assume that $\Sigma_t$ misses $\Gamma\cup {E}$.
Then the intersection $\Sigma_t\cap \partial\Omega$ is entirely contained in a
connected component $O\subset \partial\Omega\setminus(\Gamma\bigcup E)$.
By hypothesis \eqref{topASS} $O$ is simply connected. Moreover, as transversal (in fact orthogonal) intersection of $C^1$-smooth surfaces, the set $\Sigma_t\cap \partial\Omega$ is a one dimensional immersed $C^1$- submanifold without boundary, i.e. a closed curve (in two dimensions it consists of two points). Since $\Sigma_t$ has no singular points, the curve $\Sigma_t\cap \partial\Omega$ has no self-intersection and thus is a simple closed curve embedded in the simply connected subset $O$. By the Jordan curve theorem, $\Sigma_t\cap \partial\Omega$ separates $O$ in two parts, one of which, say $O_+$, is enclosed by $\Sigma_t\cap \partial\Omega$. Let $\Omega_+$ be the subset of $\Omega$ whose boundary is $\Sigma_t\bigcup O_+$, and define the new function
\begin{equation}
u_t(x):=
\left\{
\begin{array}{ll}
u(x),& x\in\Omega\setminus \Omega_+,\\
t,& x \in \overline{\Omega}_+,
\end{array}
\right.
\end{equation}Note that $u_t$ may have modified values at the boundary,
but only off the electrodes. Since $\Omega_+$ is an extension domain ($\partial\Omega_+ \cap \Omega=\Sigma_t$ has a unit normal
everywhere) the new map $u_t\in H^1(\Omega)$ and strictly decreases
the functional \eqref{1LaplFunctional} unless $u=u_t$.  This contradicts the minimizing property of $u$.
Therefore $u\equiv t$ in $\Omega_+$, which now contradicts \eqref{common|J|}.
Therefore $\Sigma_t$ intersects $\Gamma\bigcup E$, and thus $u|_{\Sigma_t}=\tilde{u}|_{\Sigma_t}=t$. Since the set $\{\Sigma_t:~ t -\text{regular value}\}$ is dense in $\overline\Omega$, the identity \eqref{u=tildeu+C} follows. Now \eqref{common|J|} yields that $\sigma=\tilde\sigma$, a.e. in $\Omega$, and by continuity in $\Omega$.

\end{proof}

\section{A minimization algorithm for the weighted gradient functional with CEM boundary constraints}\label{algorithm}
In this section we propose an iterative algorithm which minimizes the functional $G_a$ in \eqref{1LaplFunctional}. It is the analogue of an algorithm in \cite{NTT09} adapted to the CEM boundary conditions.

The following lemma is key to constructing a minimizing sequence for the functional $G_a$.
\begin{lemma}\label{decreasingG_a}Assume that $v\in H^1(\Omega)$ satisfies
\begin{align}
\epsilon\leq \frac{a}{|\nabla v|}\leq \frac{1}{\epsilon},
\end{align}for some $\epsilon>0$, and let $(u,U)\in H^1(\Omega)\times\Pi$ be the unique solution to the forward problem for $\sigma:=a/|\nabla v|$. Then
\begin{align}\label{decreasing_1LaplaceFunctional}
G_a(u,U)\leq G_a(v,V), \quad\text{for all } V\in\Pi.
\end{align}
Moreover, if equality holds in \eqref{decreasing_1LaplaceFunctional} then $(u,U)=(v,V)$.
\end{lemma}
\begin{proof}
Let $V\in\Pi$ be arbitrary. Since $(u,U)$ is a global minimizer of $F_\sigma$ as in \eqref{fowardFunctional} with $\sigma=a/|\nabla v|$ as shown in Theorem \ref{existence}, we have  the inequality:
\begin{align}
G_a(v,V)&=\int_\Omega a|\nabla v|dx + \frac{1}{2}\sum_{k=0}^N\left[ \int_{e_k} \frac{1}{z_k}(v-V_k)^2d s- 2I_k V_k\right]\nonumber\\
&=\frac{1}{2}\int_\Omega a|\nabla v|dx + F_{\frac{a}{|\nabla v|}}(v,V)\nonumber\\
&\geq\frac{1}{2}\int_\Omega a|\nabla v|dx + F_{\frac{a}{|\nabla v|}}(u,U).\label{estim1}
\end{align}
Writing
\begin{align*}
\int_\Omega a|\nabla u|dx&=\int_\Omega \left[\frac{a}{|\nabla v|}\right]^{\frac{1}{2}}|\nabla v|\left[\frac{a}{|\nabla v|}\right]^{\frac{1}{2}}|\nabla u|dx\\ &\leq\left(\int_\Omega \frac{a}{|\nabla v|}|\nabla v|^2dx\right)^{\frac{1}{2}}\left(\int_\Omega\frac{a}{|\nabla v|}|\nabla u|^2dx\right)^{\frac{1}{2}}\\
&\leq\frac{1}{2}\int_\Omega a|\nabla v|dx + \frac{1}{2} \int_\Omega \frac{a}{|\nabla v|}|\nabla u|^2dx,
\end{align*}we also obtain

\begin{align}
G_a(u,U)&=\int_\Omega a|\nabla u|dx +\frac{1}{2}\sum_{k=0}^N\left[\int_{e_k} \frac{1}{z_k} (u-U_k)^2d s- 2I_k U_k\right]\nonumber\\
&\leq\frac{1}{2}\int_\Omega a|\nabla v|dx + \frac{1}{2} \int_\Omega \frac{a}{|\nabla v|}|\nabla u|^2dx + \frac{1}{2}\sum_{k=0}^N\left[ \int_{e_k} \frac{1}{z_k} (u-U_k)^2d s- 2I_k U_k\right]\nonumber\\
&=\frac{1}{2}\int_\Omega a|\nabla v|dx+ F_{\frac{a}{|\nabla v|}}(u,U).\label{estim2}
\end{align}

From \eqref{estim1} and \eqref{estim2} we conclude \eqref{decreasing_1LaplaceFunctional}.
Moreover, if the equality holds in \eqref{decreasing_1LaplaceFunctional} then equality holds in \eqref{estim1}, and thus
\begin{align}\label{alsoMinimizer}
F_{a/|\nabla v|}(u,U)=F_{a/|\nabla v|}(v,V).
\end{align}
Since $(u,U)$ is a solution to the forward problem (for $\sigma=a/|\nabla v|$) it is also a
global minimizer of $F_{a/|\nabla v|}$ over $H^1(\Omega)\times \Pi$. But \eqref{alsoMinimizer} shows that
$(v,V)$ is also a global minimizer for $F_{a/|\nabla v|}$. Now the uniqueness of the global minimizers in Theorem \ref{existence} (for $\sigma=a/|\nabla v|$) yields $(u,U)=(v,V)$.

\end{proof}

{\bf Algorithm:} We assume the magnitude $a$ of the current density satisfies
\begin{align}
\mbox{essinf}(a)>0.
\end{align}

Let $\epsilon>0$ be the lower bound in \eqref{lowerbound_sigma}, and $\delta>0$ a measure of error to be used in the stopping criteria.

\begin{itemize}
\item Step 1: Solve (\ref{conductivity_eq}, \ref{robin_kForward}, \ref{inject_kForward}) and (\ref{no_flux_off_electrodesForward}) for $\sigma=1$, and let $u_0$ be its  unique solution. Define
    $$\sigma_1:=\min\left\{\max\left\{\frac{a}{|\nabla u_0|},\epsilon\right\},\frac{1}{\epsilon}\right\};$$
\item Step 2: For $\sigma_n$ given: Solve (\ref{conductivity_eq}, \ref{robin_kForward}, \ref{inject_kForward}) and (\ref{no_flux_off_electrodesForward}) for the unique solution $u_{n}$;
\item Step 3: If\[\|\nabla u_n- \nabla u_{n-1}\|_{C(\overline\Omega)}>\delta \frac{\epsilon}{\mbox{essinf}a},\]then define
\begin{align}\label{stop}
\sigma_{n+1}:=\min\left\{\max\left\{\frac{a}{|\nabla u_{n}|},\epsilon\right\},\frac{1}{\epsilon}\right\}
\end{align}and repeat Step 2;
\item Else STOP.
\end{itemize}

\section{Numerical Implementations}
We illustrate the theoretical results on a numerical simulation in two dimensions.

\subsection{An algorithm for the forward problem}
Given a current pattern $I \in \Pi$ and a set of surface electrodes $e_k$ with impedances $z_k$ (taken to be constant)for $k = 0, 1, \ldots, N$ satisfying \eqref{common_lower_impedance}, our iterative
algorithm consists in solving the forward problem \eqref{conductivity_eq}, \eqref{robin_kForward}, \eqref{inject_kForward}and \eqref{no_flux_off_electrodesForward} for an updated  conductivity at each step.

A piecewise linear (spline) approximation of the solution to the forward problem is sought on an uniform triangulation of the unit box $[0,1]\times[0,1]$ as shown in figure \ref{fig_UnitBox}.

\begin{figure}[ht]\centerline{\vbox{\hbox{\includegraphics[width=6cm]{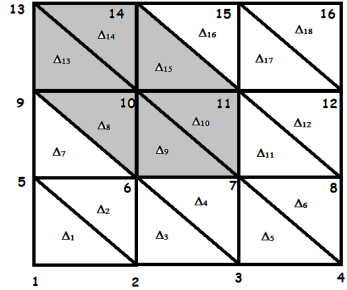}}}} \caption{The uniform triangulated unit box with $m=16$ nodes and $18$ triangles. The shaded triangles are adjacent to the $10$-th node. Note that the corner node 1 has one adjacent triangle, and nodes 4, 13 and 16 have only two adjacent triangles.}\label{fig_UnitBox}\end{figure}

For a (square) number $m$ of grid nodes, let $T_l$ be the set of planes supported in the $l$-th triangle $\Delta_l$, for $l = 1, 2, \ldots, 2(\sqrt{m}-1)^2$. More precisely, when $l= odd$,
\begin{equation}\nonumber
T_{l} = \left\{1 - \frac{1}{h}(x - x_{k_l}) - \frac{1}{h}(y - y_{k_l}), \frac{1}{h}(x - x_{k_l}), \frac{1}{h}(y - y_{k_l})\right\},\; (x,y)\in \Delta_l,
\end{equation}
where $(x_{k_l}, y_{k_l})$ is the southwest grid point of the square in which $\Delta_l$ is inscribed, and $h$ is the length of the side of the square.

When $l=even$,
\begin{equation}\nonumber
T_{l} = \left\{1 + \frac{1}{h}(x - x_{r_l}) + \frac{1}{h}(y - y_{r_l}), -\frac{1}{h}(y - y_{r_l}), -\frac{1}{h}(x - x_{r_l})\right\},\;(x,y) \in \Delta_l,
\end{equation}
where $(x_{r_l}, y_{r_l})$ is the northeast grid point of the square in which $\Delta_l$ is inscribed, and $h$ is the length of the side of the square. For example, in Figure \ref{fig_UnitBox} the triangle $\Delta_{11}$ lies in a square whose southwest grid point position is $(x_7, y_7)$ and the northeast grid point location is $(x_{12},y_{12})$.

We seek an approximation to the solution of the forward problem \eqref{conductivity_eq}, \eqref{robin_kForward}, \eqref{inject_kForward}, and \eqref{no_flux_off_electrodesForward} in the form
\begin{equation}\label{equation:assume_u_p}
u(x,y) \approx \sum_{j=1}^m u_j\psi_j(x,y),
\end{equation}
where $\psi_j$ is the sum over those planes $T_l$'s, that are adjacent to the $j$-th node in the unit box, see figure \ref{fig_UnitBox}. By substituting (\ref{equation:assume_u_p}) into (\ref{Gateaux}), and by selecting $v = \psi_j$, for  $j = 1, \ldots, m$, and $V \equiv \vec{0}$, we get the set of equations
\begin{equation}\label{equation:system1}
\int_\Omega \sigma\nabla u\cdot \nabla \psi_j dxdy + \sum_{k=0}^N\frac{1}{z_k}\int_{e_k}(u-U_k)\psi_j ds = 0, \quad \forall j = 1, \ldots, m,
\end{equation}
which is augmented with the second set of equations
\begin{equation}\label{equation:system2}
-\sum_{k=0}^N\frac{1}{z_k}\int_{e_k}(u-U_k)V_k^j ds = \sum_{k=0}^NI_kV_k^j, \quad \forall j = 0, \ldots N-1
\end{equation} obtained by setting $v\equiv 0$ and $V^j_k = 1$ whenever $k = j$ for $k = 0, \ldots, N-1$, and $V^j_N = -1$ for $j = 0, \ldots N-1$.

Note that forming $V$ in this fashion is equivalent to choosing for each $j = 0, 1, \ldots, N-1$ the $j$-th vector for the $j$-th equation in (\ref{equation:system2}) from the set
\begin{equation}\nonumber
\left\{
\left[
\begin{array}{c}
1 \\ 0 \\ 0 \\ \vdots \\ 0 \\ -1
\end{array}
\right],
\left[\begin{array}{c}
0 \\ 1 \\ 0 \\ \vdots \\ 0 \\ -1
\end{array}
\right],
\cdots,
\left[
\begin{array}{c}
0 \\ 0 \\ \vdots \\ 0 \\ 1 \\ -1
\end{array}
\right]
\right\},
\end{equation}
which is a basis for $\Pi$,  and then setting
\begin{equation}\nonumber
U_N = -\sum_{k=0}^{N-1} U_k.
\end{equation}
The values of $\{u_1, u_2, \ldots, u_m, U_0, U_1, \ldots, U_{N-1}\}$ are then solutions to the linear system
\begin{equation}\label{equation:linear_system}
\left[
\begin{array}{cc}
\Lambda & \Psi \\
\Psi^T & \Upsilon
\end{array}
\right]
\left[
\begin{array}{c}
u_1 \\ u_2 \\ \vdots \\ u_m \\
U_0 \\ U_1 \\ \vdots \\ U_{N-1}
\end{array}
\right] =
\left[
\begin{array}{c}
0 \\ 0 \\ \vdots \\ 0 \\
I_0 - I_N \\ I_1 - I_N \\ \vdots \\ I_{N-1} - I_N
\end{array}
\right]
\end{equation}
where the entries of $\Lambda$ are
\begin{equation}\nonumber
\Lambda_{i, j} = \int_\Omega \sigma\nabla\psi_j \cdot \nabla \psi_i dxdy + \sum_{k=0}^N\frac{1}{z_k}\int_{e_k}\psi_j\psi_i ds, \quad i,j = 1, 2, \ldots, m,
\end{equation}
the entries of $\Psi$ are
\begin{equation}\nonumber
\Psi_{i, j} = \frac{1}{z_N}\int_{e_N}\psi_i ds -\frac{1}{z_k}\int_{e_k}\psi_i ds, \quad  i = 1, 2, \ldots, m,\& ~ k = 0, \ldots, N-1,
\end{equation}
and
\begin{equation}\nonumber
\Upsilon =
\left[
\begin{array}{cccc}
\frac{|e_0|}{z_0} + \frac{|e_N|}{z_N} & \frac{|e_1|}{z_1} &\cdots & \frac{|e_{N-1}|}{z_{N-1}}\\
\frac{|e_0|}{z_0} &\frac{|e_1|}{z_1} + \frac{|e_N|}{z_N} & \cdots & \frac{|e_{N-1}|}{z_{N-1}} \\
\vdots & \vdots & \ddots & \vdots \\
\frac{|e_0|}{z_0} &\cdots & \frac{|e_{N-2}|}{z_{N-2}} &\frac{|e_{N-1}|}{z_{N-1}}+\frac{|e_N|}{z_N}
\end{array}
\right].
\end{equation}In the matrix above $|e_j|$denotes the surface area of the electrode, for $j=0,...,N$.

For other numerical schemes for solving the forward problem we refer to \cite{Vauhkonen99}.

\subsection{Simulating the interior data}

We consider a simulated planar conductivity  $\sigma$ (which models the cross section of a torso) embedded in the unit box $[0,1]\times[0,1]$; see Figure \ref{fig:torso_truth} on the left. The values of the conductivity range from $1.0 ~S/m$ to $1.8~S/m$.
\begin{figure}[ht]\label{fig:torso_truth}
\centerline{\vbox{\hbox{\epsfxsize 6cm\epsfbox{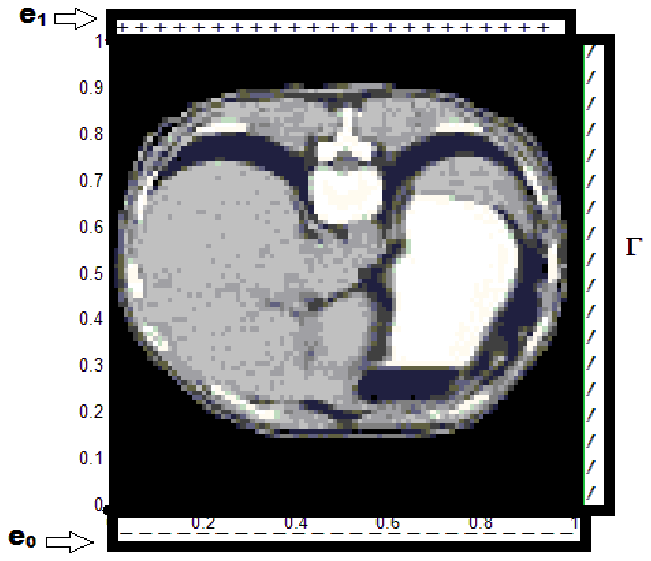}\epsfxsize
7cm\epsfbox{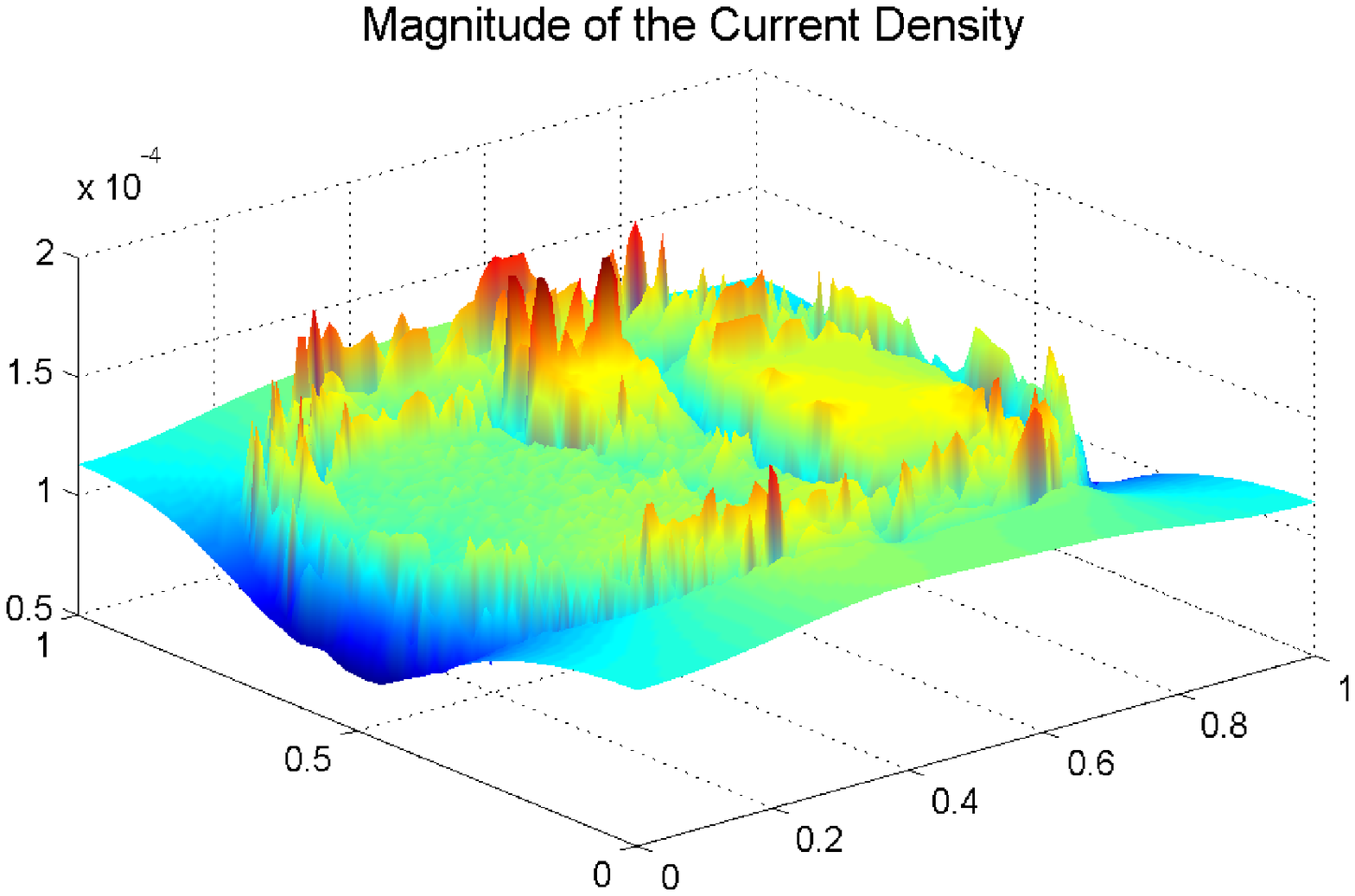}}}} \caption{The exact conductivity with the electrode set up (left).
The simulated magnitude $a$ of the current density field (right).}
\end{figure}

Two currents $-I_0=I_1 = 3~mA $ are respectively injected/extracted through the electrodes $$e_0 = \left\{(x,y)\in [0,1]\times [0,1] : ~y = 0 \right\} \quad \mbox{and} \quad e_1 = \left\{(x,y)\in [0,1]\times[0,1]:~y = 1\right\}$$
of equal impedances $z_0 = z_1 = 8.3~m\Omega \cdot m^2$.

For the given $\sigma$ we solve  the forward problem \eqref{conductivity_eq}, \eqref{robin_kForward}, \eqref{inject_kForward},\eqref{no_flux_off_electrodesForward} for $(u,U)$.  The interior data of the magnitude $a$ of the current density field is defined by $a:=\sigma|\nabla u|$; see Figure \ref{fig:torso_truth} on the right.

\subsection{Numerical reconstruction of a simulated conductivity}

Knowing the injected currents $I_0$ and $I_1$, the electrode impedances $z_0$ and $z_1$, and the corresponding magnitude $a$ of the current density  we find an approximate minimizer of $G_a$ via the iterative algorithm in section \ref{algorithm}. The iterations start with the guess $\sigma_0\equiv 1$. An approximate solution $v$ is computed on a $90 \times 90$ grid.  The stopping criterion \eqref{stop} for this experiment used $\delta = 10^{-7}$, and was attained with 320 iterations. An intermediate conductivity  $\sigma_v:=a/|\nabla v|$  is computed using the computed minimizer $v$, see Figure \ref{fig:min_sigma}.

\begin{figure}[h]
\begin{center}
\resizebox{10cm}{!}{\includegraphics{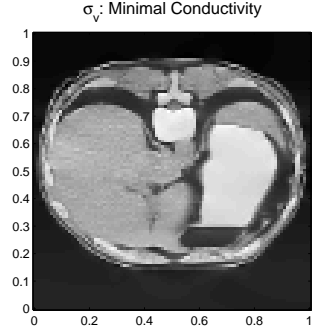}}%
\caption{Without any boundary measurement an intermediate conductivity can be recovered}%
\label{fig:min_sigma}
\end{center}
\end{figure}

We note that in this example $|J|\geq 0.5 >0$ everywhere, thus all the level sets of $u$ are connected. It follows from the arguments in Section 3 that there is a unique $\varphi$ such that $u(x)=\varphi(v(x))$ and
\begin{align*}
\sigma(x)=\frac{1}{\varphi'(v(x))}\sigma_v(x),\quad x\in\Omega.
\end{align*}
The function $\varphi$ can be determined from knowledge of $u$ on the curve $\Gamma = \{(1,y): 0\leq y\leq 1\}$, which connects the two electrodes. More precisely, for each point on  $\Gamma$ the function $\varphi$ maps the computed value of $v$ to the measured value of $u|_\Gamma$ at the same point.  Figure \ref{fig:voltage_gamma} on the left shows the resulting $\varphi$ on the range $v(\Gamma)$.  Figure \ref{fig:voltage_gamma} on the right shows  $1/\varphi'\circ v$, which is the scaling factor needed to obtain the true conductivity $\sigma$.

\begin{figure}[h]
\begin{center}
\resizebox{8cm}{!}{\includegraphics{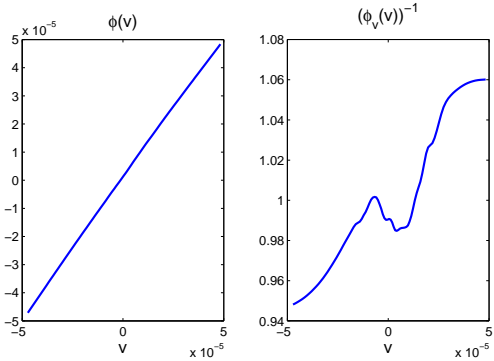}}
\caption{Graph of the function $\varphi$  (left) vs. graph of $1/(\varphi'\circ v)$  (right).}
\label{fig:voltage_gamma}
\end{center}
\end{figure}

\begin{figure}[ht]
\centerline{\vbox{\hbox{\epsfxsize 10cm\epsfbox{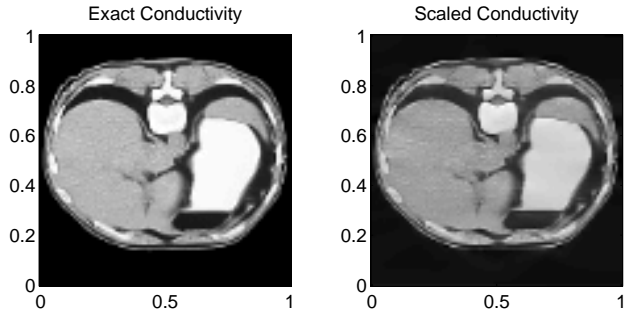}}}}
\caption{Exact conductivity (left) vs. Reconstructed conductivity (right)}
\label{conductivities6.5}
\end{figure}

In Figure \ref{conductivities6.5} the reconstructed conductivity $\sigma$ is shown on the right against the exact conductivity on the left. The $L_2$ error of the reconstruction is $0.04$.

\section*{Acknowledgments} We are grateful to the anonymous referees for their valuable comments.
In particular one of the comments uncovered a gap in a previous version of the manuscript,
and another comment suggested that our arguments would also work for the Neumann problem (see
the remark at the end of Section 3). The work of
A. Tamasan has been supported by the NSF Grant DMS-1312883, as was that of J. Veras as part of his
Ph.D. research at the University of Central Florida. The work of A. Nachman
has been supported by the NSERC Discovery Grant 250240.

\appendix

\section{A minimization approach for the Complete Electrode Model}\label{variationalForward}
\setcounter{section}{1}

In this appendix we show solvability of the forward problem for the Complete Electrode Model of
\cite{somersaloCheneyIsaacson} by recasting it into a minimization problem.
While this approach is less general than the one given in \cite{somersaloCheneyIsaacson}
(we assume a real valued conductivity and positive electrode impedances), it explains
how we are led to introduce the functional \eqref{1LaplFunctional} in the solution of the inverse problem.
For existence and uniqueness of solutions of the forward problem,
the conductivity and electrode impedances need not be smooth: $\sigma\in L^\infty(\Omega)$ and $z_k\in L^\infty(e_k)$
satisfy
\begin{align}\label{sigma,z_k>0}
\text{essinf}_{\Omega}\sigma \geq \epsilon>0,\quad\text{essinf}_{e_k}{z_k}\geq\epsilon>0,~~k=0,...,N.
\end{align}
Let $H^1(\Omega)$ be the space of functions which together with their gradients lie in $L^2(\Omega)$, and $\Pi$ be the hyperplane in \eqref{normalization}. We seek weak solutions to \eqref{conductivity_eq}, \eqref{robin_kForward} \eqref{inject_kForward}, \eqref{no_flux_off_electrodesForward}, and \eqref{sum0inject} in the Hilbert space $H^1(\Omega)\times \Pi$, endowed with the product
$$\langle(u, U),(v,V)\rangle:=\int_{\Omega} uvdx +\int_\Omega \nabla u\cdot\nabla v dx +\sum_{k=0}^NU_kV_k,$$
and the induced norm
\begin{align}\label{norm}
\|(u,U)\|:=\langle (u,U),(u,U)\rangle^{1/2}.
\end{align}

We'll need the following variant of the Poicar\'e inequality, suitable for the complete electrode model.
\begin{proposition}\label{lemma_Coercivity}Let $\Omega\subset\mathbb{R}^n$, $n\geq2$, be an open,
connected, bounded domain with Lipschitz boundary $\partial\Omega$, and $\Pi$ be the hyperplane in \eqref{normalization}. For $k=0,...,N$, let $e_k\subset\partial\Omega$
be disjoint subsets of the boundary of positive induced surface measure: $|e_k|>0$.

There exists a constant $C>0$, dependent only on $\Omega$ and the $e_k$'s, such that for all $u\in H^1(\Omega)$ and all $U=(U_0,...,U_N)\in \Pi$, we have
\begin{align}\label{Poincare}
\int_{\Omega} u^2dx+\sum_{k=0}^N U_k^2\leq C\left(\int_{\Omega}|\nabla u|^2dx+\sum_{k=0}^N \int_{e_k}(u-U_k)^2ds\right).
\end{align}
\end{proposition}
\begin{proof} We will show that
\begin{align}\label{infimum_coercivity}
\inf_{ (u,U)\in H^1(\Omega)\times\Pi}\frac{\int_{\Omega}|\nabla u|^2dx+\sum_{k=0}^N \int_{e_k}(u-U_k)^2ds}{\int_\Omega |\nabla u|^2dx+\int_{\Omega} u^2dx+\sum_{k=0}^N U_k^2}=:\kappa >0.\end{align}

We reason by contradiction: Assume the infimum in \eqref{infimum_coercivity} is zero. Without loss of generality (else normalize to 1), there exists a sequence $\{(u_n,U^n)\}$ in the unit sphere of $ H^1(\Omega)\times\Pi$,
$\|(u_n,U^n)\|=1$, and such that
\begin{align}
&0=\lim_{n\to\infty}\int_{\Omega}|\nabla u_n|^2dx\label{gradsTo0},\\
&0=\lim_{n\to\infty}\int_{e_k}(u_n-U_k^n)^2, \quad\text{for } k=0,...,N. \label{converg_on_electrode}
\end{align}
Due to the compactness of the unit sphere in $\Pi$ and of the weakly compactness of the unit sphere in $H^1(\Omega)$ it follows that there exists some $(u_*,U^*)\in H^1(\Omega)\times\Pi$ with
\begin{align}\label{limitOnSphere}
\|(u_*,U^*)\|=1,
\end{align}such that, on a subsequence (relabeled for simplicity),
\begin{align}
&u_n\rightharpoonup u_*\quad\text{in } H^1(\Omega),\label{weakConvergence}\\
&U^n\to U^*\quad\text{in }\Pi,\quad\text{as } n\to\infty.\label{convergence_inU}
\end{align}

Since the sequence $\{u_n\}$ is bounded in $H^1(\Omega)$, the trace theorem implies that $u_n|_{e_k}$ is (uniformly in $n$) bounded in $H^{1/2}(e_k)$, hence also in $L^1(e_k)$, for each $k=0,...,N$.
Using \eqref{converg_on_electrode} and \eqref{convergence_inU} in
\begin{align*}
\int_{e_k} (u_n- U^*_k)^2 ds=&\int_{e_k} (u_n - U^n_k)^2 ds + 2(U^n_k-U^*_k)\int_{e_k}u_n ds\\
 &+|e_k|\left[(U^*_k)^2-(U^n_k)^2)\right],
\end{align*}we obtain
$u_n|_{e_k}\to U^*_k\quad \text{in }L^2(e_k)$. Since $u_n|_{e_k}\rightharpoonup u_*|_{e_k}$, we conclude that
\begin{align}\label{electrode_Conclusion}
u_*|_{e_k}=U^*_k\quad \text{for each } k=0,...,N.
\end{align}

Now using \eqref{gradsTo0} and \eqref{weakConvergence}
\begin{align*}
0\leq &\int_\Omega |\nabla(u_n-u_*)|^2dx=\int_{\Omega}|\nabla u_n|^2dx-2\int_{\Omega} \nabla u_n\cdot\nabla u_*+\int_\Omega|\nabla u_*|^2dx\\
&\longrightarrow-\int_\Omega |\nabla u_*|^2 dx,\quad\text{as }n\to\infty,
\end{align*}
and, since $\Omega$ is connected,
\begin{align}\label{u*=const.}
u_*\equiv const. \;\text{in }\overline{\Omega}.
\end{align}

From \eqref{electrode_Conclusion} and \eqref{u*=const.} we conclude that $u_*$ restricts to the same constant on each electrode, and thus
$U^*_0=U^*_1=...=U^*_N=u_*.$ Since $U^*\subset \Pi$, we must have $U^*=\langle0,...,0\rangle$ and then $u_*\equiv0$, thus contradicting \eqref{limitOnSphere}.

\end{proof}

\begin{proposition}\label{quadraticProperties} Let $\Omega$, $\Pi$, and $e_k\subset \partial\Omega$, $k=0,...,N$ be as in Proposition \ref{lemma_Coercivity}.
For $\sigma$, and $z_k$, $k=0,...,N$ satisfying \eqref{sigma,z_k>0}, and $I=(I_0,...,I_N)\in\mathbb{R}^{N+1}$, let us consider the quadratic functional $F_\sigma:H^1(\Omega)\times\Pi\to\mathbb{R}$ defined by
\begin{align}\label{fowardFunctional}
F_\sigma(u,U):=\frac{1}{2}\int_{\partial\Omega}\sigma|\nabla u|^2dx +\frac{1}{2}\sum_{k=0}^N\int_{e_k}\frac{1}{z_k}(u-U_k)^2ds-\sum_{k=0}^N I_kU_k.
\end{align}
Then

(i) $F_\sigma$ is strictly convex

(ii) $F_\sigma$ is Gateaux differentiable in $H^1(\Omega)\times\Pi$, and the derivative at $(u,U)$ in the direction $(v,V)$ is given by
\begin{align}\label{Gateaux}
\langle DF_\sigma(u,U);(v,V)\rangle&=\int_{\Omega}\sigma\nabla u\cdot \nabla vdx +\sum_{k=0}^N\int_{e_k}\frac{1}{z_k}(u-U_k)(v-V_k)ds\nonumber\\&\quad -\sum_{k=0}^NI_kV_k
\end{align}
(iii) $F_\sigma$ is coercive, more precisely,
\begin{align}\label{coerc_Expanded}
F_\sigma(u,U)\geq \frac{c}{2}\|(u,U)\| - \frac{1}{2c}\sum_{k=0}^N I_k^2,
\end{align}
for some constant $c>0$ dependent on the lower bound  $\epsilon$ in \eqref{sigma,z_k>0}, and $\kappa$ in \eqref{infimum_coercivity}.
\end{proposition}
\begin{proof}(i) The functional has two quadratic terms, each strictly convex, and one linear term, hence the sum is strictly convex.
(ii) The Gateaux differentiability and the formula \eqref{Gateaux} follow directly from the definition of $F_\sigma$.

(iii) Proposition \ref{lemma_Coercivity} above shows that
\begin{align*}
F_\sigma(u,U)\geq c\|(u,U)\|^2-\sum_{k=0}^NI_kU_k,
\end{align*}where
$c=\frac{\kappa}{2} \epsilon >0.$
By completing the square one obtains
\begin{align*}
F_\sigma(u,U)&\geq c\|u\|_{H^1(\Omega)}^2+c\sum_{k=0}^N\left(U_k-\frac{I_k}{2c}\right)^2 -\frac{1}{4c}\sum_{k=0}^N I_k^2\\
&\geq c\|u\|_{H^1(\Omega)}^2+c\sum_{k=0}^N\left(\frac{1}{2}U_k^2-\frac{I_k^2}{4c^2}\right) -\frac{1}{4c}\sum_{k=0}^N I_k^2\\
&\geq \frac{c}{2}\|(u,U)\|^2-\frac{1}{2c}\sum_{k=0}^N I_k^2
\end{align*}\end{proof}

The proposition below revisits \cite[Proposition 3.1.]{somersaloCheneyIsaacson} and separates the role of the conservation of charge condition \eqref{sum0inject}. This becomes important in our  minimization approach, where we shall see that $F_\sigma$ has a unique minimizer independently of the condition of \eqref{sum0inject} being satisfied. However, it is only for currents satisfying \eqref{sum0inject}, that the minimizer satisfies \eqref{inject_kForward}. This result does not use the reality of $\sigma$ and of $z_k$'s. Recall that the Gateaux derivative of $DF_\sigma$ is given in \eqref{Gateaux}.

\begin{proposition}\label{carification} Let $\Omega$, $\Pi$,  $e_k\subset \partial\Omega$, $z_k$, $k=0,...,N$, and $\sigma$ be as in Proposition \ref{quadraticProperties}.

(i) If $(u,U)\in H^1(\Omega)\times\Pi$ is a weak solution to \eqref{conductivity_eq}, \eqref{robin_kForward}, \eqref{inject_kForward} and \eqref{no_flux_off_electrodesForward}, then \eqref{sum0inject} holds and
\begin{align}\label{weakformulationEquivalence}
 \left\langle DF_\sigma(u,U);(v,V)\right\rangle=0,\quad\forall \;(v,V)\in H^1(\Omega)\times\Pi.
\end{align}

(ii) If $(u,U)\in H^1(\Omega)\times\Pi$ satisfies \eqref{weakformulationEquivalence}, then it solves \eqref{conductivity_eq},  \eqref{robin_kForward} and \eqref{no_flux_off_electrodesForward}. In addition, if $I_k$'s satisfy \eqref{sum0inject}, then  \eqref{inject_kForward} also holds.

\end{proposition}
\begin{proof}
(i) Follows from a direct calculation and Green's formula.

(ii) Assume that \eqref{weakformulationEquivalence} holds.

By choosing $v\in H^1_0(\Omega)$ arbitrary and $V=\overrightarrow{0}$ in \eqref{weakformulationEquivalence} we see that
$$\int_\Omega\sigma\nabla u\cdot\nabla vdx=0.$$ Thus $u\in H^1(\Omega)$ is a weak solution of \eqref{conductivity_eq}.

For each fixed $k=0,1,...,N$ keep $V=\overrightarrow{0}$ as above, but now choose $v\in H^1(\Omega)$ arbitrary with $v|_{\partial\Omega\setminus e_k}=0$. A straightforward calculation starting from \eqref{weakformulationEquivalence} shows that
$$\int_{e_k}\frac{1}{z_k}\left( u-U_k+z_k\sigma\frac{\partial u}{\partial \nu}\right)vds =0.$$
Since $v|_{e_k}$ were arbitrary \eqref{robin_kForward} follows.

Now choose $V=\overrightarrow{0}$ as above but $v\in H^1(\Omega)$ arbitrary with $v|_{e_k}=0$ for all $k=0,...,N$. It follows from \eqref{weakformulationEquivalence} that $$\int_{\partial\Omega}\sigma\frac{\partial u}{\partial\nu} v ds=0.$$
Since the trace of $v$ is arbitrary off the electrodes \eqref{no_flux_off_electrodesForward} holds.

Finally, for an arbitrary $V\in\Pi$ choose $v\in H^1(\Omega)$ with the trace
$v=V_k$ on each $e_k$, $k=0,...,N$ and  $v=0$ off the electrodes.
By using the already established relations \eqref{conductivity_eq}, \eqref{robin_kForward}, \eqref{no_flux_off_electrodesForward} and Green's formula in \eqref{weakformulationEquivalence} we obtain
$$\sum_{k=0}^NV_k\left(\int_{e_k}\sigma\frac{\partial u}{\partial\nu}ds-{I_k}\right)=0.$$
On the one hand, by introducing the notation $\overrightarrow\alpha:=\langle\alpha_0,...,\alpha_N\rangle$ with
$$\alpha_k:=\int_{e_k}\sigma\frac{\partial u}{\partial\nu}ds-{I_k},\quad k=0,...,N,$$
we just showed that $\overrightarrow\alpha\perp\Pi$.  Note that so far we have not used the conservation of charge condition \eqref{sum0inject}.

On the other hand, by using \eqref{no_flux_off_electrodesForward}, \eqref{sum0inject}, and \eqref{conductivity_eq}  in the divergence formula, we have
\begin{align*}
\sum_{k=0}^N\alpha_k=\int_{\partial\Omega}\sigma\frac{\partial u}{\partial\nu}ds=\int_\Omega\nabla\cdot \sigma\nabla u dx=0,
\end{align*}which yields $\overrightarrow\alpha\in \Pi$. Therefore $\overrightarrow\alpha\in\Pi^\perp\cap\Pi=\overrightarrow 0$, and \eqref{inject_kForward} holds.

\end{proof}
The following result establishes existence and uniqueness of the weak solution to the foward CEM problem; contrast with the proof of Theorem 3.3 in \cite{somersaloCheneyIsaacson}.
\begin{theorem}\label{existence}
Let $\Omega$, $\Pi$, $e_k\subset \partial\Omega$, $z_k$, for $k=0,...,N$, and $\sigma$  be as in Proposition \ref{quadraticProperties}. Let $F_\sigma:H^1(\Omega)\times\Pi\to\mathbb{R}$ be defined in \eqref{fowardFunctional}.

(i) Then $F_\sigma$ has a unique minimizer $(u,U)\in H^1(\Omega)\times\Pi$. If, in addition, the injected currents $I_k$'s satisfy \eqref{sum0inject} the minimizer is the weak solution of the problem \eqref{conductivity_eq}, \eqref{robin_kForward}, \eqref{inject_kForward}, and \eqref{no_flux_off_electrodesForward}.

(ii) If the problem \eqref{conductivity_eq}, \eqref{robin_kForward}, \eqref{inject_kForward},\eqref{no_flux_off_electrodesForward} has a solution, then it is a minimizer of $F_\sigma$ in the whole space $H^1(\Omega)\times\Pi$ and hence unique. Moreover, the current $I_k$'s satisfy \eqref{sum0inject}.
\end{theorem}
\begin{proof}
(i) Let $$d=\inf_{H^1(\Omega)\times\Pi} F_\sigma(u,U),$$ and consider a minimizing sequence $\{(u_n,U^n)\}$ in $H^1(\Omega)\times\Pi$,
\begin{align}
d\leq F_\sigma(u_n,U^n)\leq d+\frac{1}{n}.
\end{align}
Since
$\inf F_\sigma\geq-\frac{1}{4c}\sum_{k=0}^N I_k^2
$ we have $d\neq-\infty$. Following \eqref{coerc_Expanded},
$$\lim_{\|(u,U)\|\to\infty}F_\sigma(u,U)=\infty.$$ Thus the minimizing sequence  must be bounded, hence weakly compact. In particular, for a subsequence (relabeled for simplicity) there is some $(u_*,U^*)\in H^1(\Omega)\times\Pi$, such that
\begin{align}\label{weakConvMinSequence}
u_n\rightharpoonup u_*\;\text{in }H^1(\Omega),\quad \text{and }U_n\to U^*\;\text{in }\Pi,\;\text{as }n\to\infty.
\end{align}
On the other hand since $F_\sigma$ is convex, and Gateaux differentiable at $(u_*,U^*)$ in the direction $(u_n-u_*,U^n-U^*)$, we have
\begin{align}
F_\sigma(u_n,U^n)\geq F_\sigma(u_*,U^*)+\langle DF_\sigma(u_*,U^*);(u_n-u_*,U^n-U^*)\rangle.
\end{align}
We take the limit as $n\to\infty$. The weak convergence in \eqref{weakConvMinSequence} yields
\begin{align*}\langle DF_\sigma(u_*,U^*),(u_n-u_*,U^n-U^*)\rangle\to 0.\end{align*}
Thus $d\geq F_\sigma(u_*, U^*)\geq d$ which shows that $(u_*,U^*)$ is a global minimizer. Strict convexity of $F_\sigma$ implies it is unique. At the minimum $(u_*,U^*)$ the Euler-Lagrange equations \eqref{weakformulationEquivalence} are satisfied. An application of Proposition \ref{carification} part (ii) shows that $(u_*,U^*)$ is a weak solution to the forward problem.

(ii) Proposition \ref{carification} part (i) shows that $(u_*,U^*)$ solves the Euler-Lagrange equations, and due to the convexity it is a minimizer of $F_\sigma$. Due to the strict convexity of the functional the minimizer is unique, hence the weak solution is unique.

\end{proof}
\section{On the regularity up to the boundary for CEM boundary conditions}

If $\sigma\in C^\alpha(\Omega)$, then interior elliptic regularity
yields $u\in C^{1,\alpha}(\Omega)$. The following result considers the regularity
up to the boundary; part a) in the proposition below was already proved in \cite[Remark 1]{DHHS}. We reproduce
the proof for the reader's convenience.
Let $E:=\bigcup_{k=0}^N e_k\subset\partial\Omega\displaystyle$ be the union set of the electrodes.

\begin{proposition}\label{regularity}
Let $\Omega\subset \mathbb{R}^n$ be a $C^2$-domain, and
let $\sigma\in C^\alpha(\overline\Omega)$ be $C^2$-smooth near the boundary.
Assume that the electrodes $e_k$ has Lipschitz boundary, and $z_k\in C^2(e_k)$, $k=0,...,N$.
Let $(u,U)\in H^1(\Omega)\times\Pi$  be the solution to the forward problem.

Then

a) $u\in H^{2-\epsilon}(\Omega)$, for all $\epsilon>0$.

 b) For any $x_0\in\partial\Omega\setminus \partial{E}$, there is a neighborhood $V_0\subset \overline\Omega$ of $x_0$,
 and a function $v_0\in H^{3-\epsilon}(\Omega)$ for all $\epsilon>0$, such that $u|_{V_0}=v|_{V_0}$.
\end{proposition}
\begin{proof}

a) Since $u\in H^1(\Omega)$ it follows from \eqref{robin_kForward} that $z_k\sigma\frac{\partial}{\partial \nu}u\in H^{1/2}(E)$. Since $1/(z_k\sigma)\in C^2(E)$, we have $\frac{\partial}{\partial \nu}u\in H^{1/2}(E)$. By \cite[Theorem 11.4]{lionsMagenes} the extension by zero to the whole boundary yields $\frac{\partial}{\partial \nu}u\in H^{1/2-\epsilon}(\partial\Omega)$,  and thus
$\sigma\frac{\partial}{\partial \nu}u\in H^{1/2-\epsilon}(\partial\Omega)$. Now apply
the elliptic regularity for the Neumann problem \cite[Remark 7.2]{lionsMagenes} to conclude
$u\in H^{2-\epsilon}(\Omega)$, for all $\epsilon>0$.

b) Let $x_0\in  E$, and choose $r_0>0$ be sufficiently small so that $\sigma$ is $C^2$-smooth in $\{x\in\overline\Omega:~|x-x_0| < 2r_0\}$ and
$\{x\in\partial\Omega:~|x-x_0| < 2r_0\}\subset E$.

Let $V_0:=\{x\in\overline\Omega:~|x-x_0| \leq r_0/2\}$. We define
$$v_0(x)=\chi_0(x) u(x),\quad x\in\overline\Omega,$$
where $\chi_0\in C^\infty(\overline\Omega)$ is the cutoff function with $\chi_0(x)= 1$, if $x\in V_0$, and $\chi_0(x)=0$ if  $\{x\in\overline\Omega:~|x-x_0| \geq r_0\}$.
Then, by part a) we have $$\nabla\cdot\sigma\nabla v_0=u\nabla\cdot\sigma\nabla\chi_0+2\sigma\nabla \chi_0\cdot\nabla u\in H^{1-\epsilon}(\Omega),$$for all $0<\epsilon\leq 1$.

Also by part a) we have that the trace $u|_{\partial\Omega}\in H^{3/2-\epsilon}(\partial\Omega)$. Now, by  \eqref{robin_kForward},
\begin{align*}
z_k\sigma\frac{\partial v_0}{\partial\nu}&=
\chi_0z_k\sigma\frac{\partial u}{\partial\nu}+z_k\sigma u\frac{\partial\chi_0}{\partial\nu}\\
&=\chi_0(U_k-u) +z_k\sigma u\frac{\partial\chi_0}{\partial\nu}\in H^{3/2-\epsilon}(\partial\Omega),\end{align*}
and thus, $\sigma\frac{\partial v_0}{\partial\nu}\in H^{3/2-\epsilon}(\partial\Omega).$
Now apply the elliptic regularity \cite[Theorem 7.4, Remark 7.2]{lionsMagenes} to conclude $v_0\in H^{3-\epsilon}(\Omega)$.

If $x_0\in\partial\Omega\setminus\overline{E}$, the same proof holds
if we choose $r_0>0$ such that $\{x\in\partial\Omega:~|x-x_0| < 2r_0\}\cap \overline{E}=\emptyset$.
\end{proof}

\section{The connected components of almost all level sets reach the boundary}\label{reaching_the_boundary}
Let $t$ be one of the values for which the level set $\{u(x)=t\}$ is a $C^1$- smooth hypersurface (which is the case for a.e. $t$), and $\Sigma$ be one of its connected components. We show here that $\Sigma\cap\partial\Omega\neq\emptyset$. The arguments in the proof of \cite[Theorem 1.3]{NTT09} use only the interior points of $\Omega$, and thus apply to the CEM as well. We include them here for the convenience of the reader.

Arguing by contradiction, assume that $\Sigma\bigcap\partial\Omega=\emptyset$. Then
$\partial\Omega\bigcup\Sigma$ is a compact manifold with two
connected components. Using the Alexander duality theorem in
algebraic topology for $\partial\Omega\bigcup\Sigma$ (see, e.g.
Theorem 27.10 in \cite{greenbergHarper},) we have that
$R^n\setminus(\partial\Omega\bigcup\Sigma)$ is partitioned into
three open connected components:
$(R^n\setminus\overline{\Omega})\bigcup O_1\bigcup O_2$. Since
$\Sigma\subset\Omega$ we have $O_1\bigcup
O_2=\Omega\setminus\Sigma$ and then $\partial O_i\subset
\partial\Omega\bigcup\Sigma$ for $i=1,2$.

We claim that at least one of the $\partial O_1$ or $\partial O_2$
is in $\Sigma$. Assume not, i.e. for each $i=1,2$, $\partial
O_i\bigcap\partial\Omega\neq\emptyset$. Since $\partial\Omega$ is
connected (by assumption), we have that $O_1\bigcup
O_2\bigcup\partial\Omega$ is connected which implies $O_1\bigcup
O_2\bigcup(R^n\setminus\Omega)$ is also connected. By applying
once again Alexander's duality theorem for $\Sigma\subset R^n$,
we have that $R^n\setminus \Sigma$ has exactly two open
connected components, one of which is unbounded: $R^n\setminus
\Sigma =O_{\infty}\bigcup O_0$. Since $O_1\bigcup O_2\bigcup
(R^n\setminus\Omega)$ is connected and unbounded, we have
$O_1\bigcup O_2\bigcup (R^n\setminus\Omega)\subset O_\infty$,
which leaves $O_0\subset R^n\setminus (O_1\bigcup O_2\bigcup
(R^n\setminus\Omega))\subset \Sigma$. This is impossible since
$O_0$ is open and $\Sigma$ is a hypersurface. Therefore either
$O_1$ or $O_2$ or both has the boundary in $\Sigma$.

To fix ideas, consider $\partial O_1\subset\Sigma$. If this were
the case, then we claim that $u\equiv t$ in $O_1$. Indeed, since
$O_1$ is an extension domain ($\partial O_1$ has a unit normal
everywhere) the new map $\tilde u$ defined by
\begin{equation*}
\tilde u(x)=\left\{
\begin{array}{ll}
u(x),& x\in\Omega\setminus O_1,\\
t,& x \in \overline{O_1},
\end{array}
\right.
\end{equation*}
is in $H^1(\Omega)\bigcap C(\overline\Omega)$ and decreases
the functional \eqref{1LaplFunctional}, thus contradicting the
minimizing property of $u$. Therefore $u\equiv t$ in $O_1$, which makes
$|\nabla u|\equiv 0$ in $O_1$. Again we reach a contradiction
since the set of critical points of $u$ is negligible.

These contradictions followed from the assumption that $\Sigma
\bigcap\partial\Omega=\emptyset$, and therefore $\Sigma
\bigcap\partial\Omega\neq\emptyset$

\end{document}